\documentclass[12pt,oneside]{amsart}

\usepackage{amssymb,latexsym}

\usepackage{hyperref}
\hypersetup{
  colorlinks   = true, 
  urlcolor     = blue, 
  linkcolor    = blue, 
  citecolor   = red 
}
\usepackage{anysize}
\marginsize{2cm}{2cm}{2cm}{2cm}

\usepackage[pdftex]{graphicx}

\newtheorem{theorem}{Theorem}[section]
\newtheorem{proposition}[theorem]{Proposition}
\newtheorem{lemma}[theorem]{Lemma}

\newtheorem{conjecture}[theorem]{Conjecture}
\newtheorem{corollary}[theorem]{Corollary}

\theoremstyle{definition}

\theoremstyle{remark}
\newtheorem{remark}[theorem]{Remark}
\newtheorem{question}[theorem]{Question}

\numberwithin{equation}{section}

\renewcommand{\epsilon}{\varepsilon}
\renewcommand{\phi}{\varphi}
\renewcommand{\kappa}{\varkappa}

\usepackage[pagewise]{lineno}

\title{ Ivrii's conjecture for some cases in outer and symplectic billiards}

\author{Anastasiia Sharipova}

\address{Anastasiia Sharipova, Department of Mathematics, Pennsylvania State University, University Park, PA
16802, USA}

\email{sharipova.math@gmail.com}

\thanks{This work is supported by NSF grant DMS-2005444 and by the DFG CRC-TRR 191 “Symplectic structures in geometry, algebra and dynamics” (281071066).}


\begin{document}

\begin{abstract}
We give a proof for $(2n + 1,n)$ and $(2n, n-1)$-periodic Ivrii's conjecture for planar outer billiards. We also give new simple geometric proofs for the 3 and 4-periodic cases for outer and symplectic billiards, and generalize for higher dimensions in case of symplectic billiards.
\end{abstract}

\maketitle

\section{Introduction}

It was shown by Ivrii \cite{ivrii80} that Weyl's conjecture \cite{weyl1912} about distributions of eigenvalues of the Laplacian in a domain with a smooth boundary holds with assumption that the set of periodic billiard orbits in the domain has measure zero (in the phase space). For convex domains with real analytic boundaries the measure of periodic orbits is zero and the assumption holds. However, for non-analytic boundaries the question about the measure of periodic billiard orbits turned out to be hard and remains a conjecture.

\begin{conjecture}{(Ivrii,1980)}
    The set of periodic billiard orbits in a domain with a smooth boundary (not necessarily convex) has measure zero.
\end{conjecture}

Recall that the phase space consists of pairs $(x, u)$ where $x$ is a point on the boundary and $u$ is a unit vector (velocity) directed inside the domain. A periodic orbit is a finite set of points in the phase space that are mapped to each other by a billiard map, and therefore proving the zero measure of periodic orbits is equivalent to proving the zero measure of periodic points (in the phase space).

The Ivrii conjecture is hard even for a fixed period. It is clear for period 2 since for each boundary point there exists at most one 2-periodic orbit going through it, and hence the set of 2-periodic orbits has measure zero. For period 3 it is already not obvious. First, Rychlik \cite{rychlik89} proved using symbolic calculations that the set of 3-periodic orbits for a planar domain has measure zero. Stojanov \cite{stojanov91} simplified his calculations and later Wojtkovski \cite{wojtkovski94} gave a new simple proof using the Jacobi fields. Then Vorobets \cite{vorobets94} proved a 3-periodic case for domains in higher dimensions. Finally, Glutsyuk and Kudryashov \cite{glutsyuk-kudryashov12} solved the Ivrii conjecture for period 4 for planar domains. For all other periods the Ivrii conjecture is completely open. Recent results of Callis \cite{callis22} provide an alleged step towards disproving the Ivrii conjecture.

The Ivrii conjecture is also relevant in other geometries, and was studied in surfaces of constant curvatures or for other billiard dynamics. See for example \cite{blumen-kim-nance-zhar} for a proof of the conjecture for 3-periodic orbits in billiards on the hyperbolic plane, and a disproof for billiards on the 2-sphere. The conjecture was also studied for complex billiards \cite{glutsyuk-2017} and projective billiards \cite{fierobe20}.

In this paper we are interested in the similar question as the Ivrii conjecture but for outer and symplectic billiards in convex domains. The weaker version of the Ivrii conjecture claims that the set of periodic points has an empty interior (nowhere dense). In section 2 we give the basic definitions of outer billiard and prove the following theorem

\begin{theorem}
\label{star_orbits_theorem}
    The sets of $(2n + 1,n)$ and $(2n, n-1)$-periodic outer billiard orbits for a strongly convex planar domain with a smooth boundary have empty interiors.
\end{theorem}
In this paper a \emph{strongly convex} domain means a domain which has positive curvature (positive definite second fundamental form) at every boundary point, and \emph{strictly convex} means a domain such that each line segment connected two boundary points has all its points except endpoints in the interior of the domain. 

$(m, k)$-periodic means $m$-periodic orbit with winding number $k$. The statement of the Theorem \ref{star_orbits_theorem} was proven by Tumanov and Zharnitsky \cite{tuman-zhar06} for 3 and 4-periodic orbits (which are particular examples for $n = 1$ and $n = 2$) and by Tumanov \cite{tuman20} for (5, 2) and (6, 2) using approach based on exterior differential systems (EDS). 

The Ivrii conjecture for outer billiards is false if the boundary is not smooth, as there exist polygonal billiards having open sets of periodic orbits, see for example Remark 5.1 in \cite{genin-tabach07} or the proof of Proposition 5.1 in \cite{fierobe-2024}.

The idea of our proof is considering the differential of the outer billiard map after several iterations and proving that it cannot be identity for these periods. This idea was used by Tabachnikov and Genin in their paper \cite{genin-tabach07} to prove this statement for a 3-periodic case. We also give new simple geometric proofs for periods 3 and 4.

In section 3 we consider symplectic billiards and prove the Ivrii conjecture for periods 3 and 4 for any dimensions:

\begin{theorem}
    The set of 3 and 4-periodic symplectic billiard orbits in a strongly convex domain with a smooth boundary has an empty interior.
\end{theorem}

The statement was proven by Albers and Tabachnikov \cite{albers-tabach2017} for a planar case using the EDS method.

\section{Ivrii's conjecture for outer billiards}
Let $X$ be a convex body in the plane $\mathbb{R}^{2}$. An \emph{outer billiard} is a dynamical system where a point $x$ outside of $X$ is mapped to a point $y$ on the tangent line from $x$ to $X$ such that the point of tangency $z$ is a midpoint between $x$ and $y$. Since there are two tangent lines, we make always a choice such that vector $xy$ has the same direction as an orientation of $\partial X$ at a tangency point.  We denote the outer billiard map by $F$. Outer billiard map is an area-preserving twist map of the exterior of $X$ and in case of a smooth strongly convex body it is well-defined at any point. Outer billiard was introduced by Bernhard Neumann \cite{neumann-1958} as a model to study a stability problem. See \cite{tabach93} for a detailed introduction to outer billiards.


\begin{figure}[h]
    \centering
    \includegraphics[width=55mm]{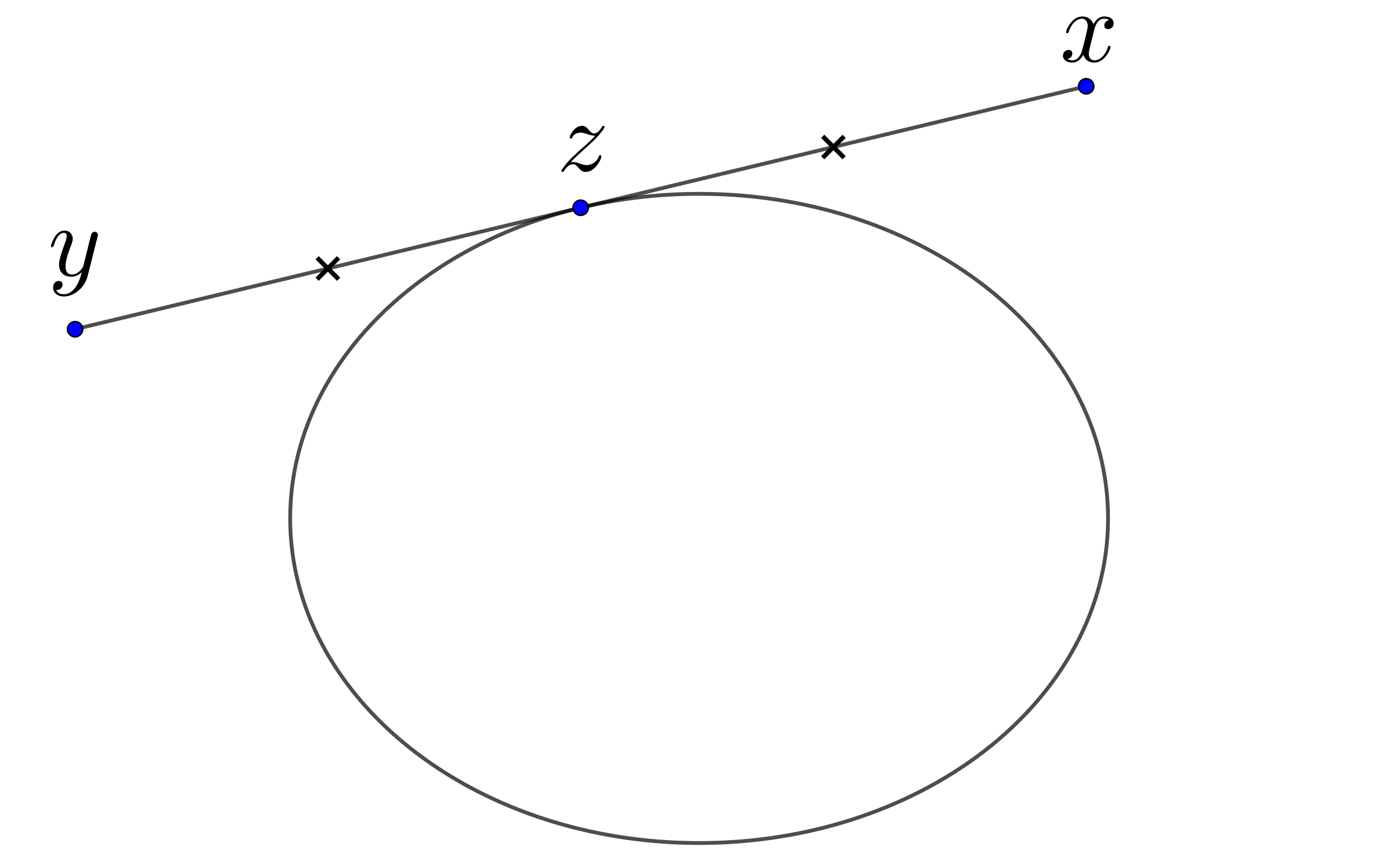}
    \caption{Outer billiard reflection law}
    \label{fig:enter-label}
\end{figure}

One can ask questions about periodic orbits for outer billiards. It is known that for any period $n$ and any rotation number there are at least two $n$-periodic orbits in case of smooth convex domains. In this work we are interested in the upper bound (in terms of measure) of them. The phase space in the case of outer billiard is the exterior of a convex body and orbit is a set of points which are mapped to each other consecutively by the outer billiard map. \emph{Outer billiard trajectory} is defined as a polygonal line consisting of segments between consecutive points of an outer billiard orbit, and we say that it is $n$-periodic if the corresponding orbit is $n$-periodic.

\subsection{New proofs for n = 3 and n = 4}

For these periods we prove an even stronger result that any tangent line contains at most two 3 and 4-periodic points. Equivalently, for any $z \in \partial X$ there is at most one 3-periodic trajectory and at most one 4-periodic trajectory passing through $z$. We always assume that our domain $X$ is strongly convex with a smooth boundary.

\begin{proposition}
\label{3-per-outer}
    For any $z \in \partial X$ there is at most one 3-periodic trajectory passing through $z$.
\end{proposition}


\begin{figure}[bh]
    \centering
    \includegraphics[width=125mm]{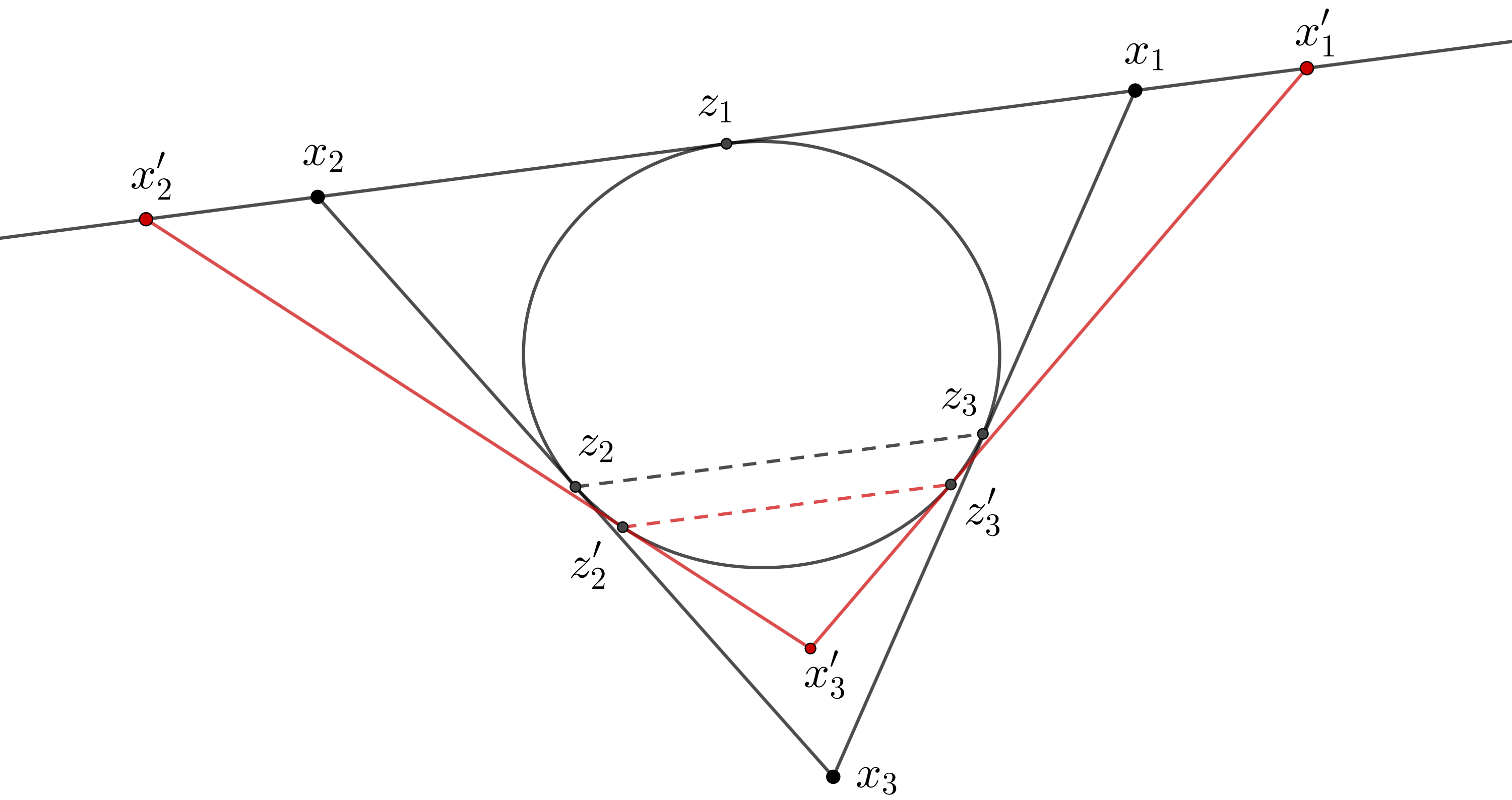}
    \caption{Two 3-periodic orbits}
    \label{3-periodic-outer-Ivrii}
\end{figure}

\begin{proof}
Let $l$ be a tangent line with tangency point $z_1 = z$ and assume we have two 3-periodic orbits $(x_1, x_2, x_3)$ and $(x_1', x_2', x_3')$ with tangency points $z_1, z_2, z_3$ and $z_1, z_2', z_3'$ as in the figure $\ref{3-periodic-outer-Ivrii}$. Since the domain is strongly convex, it lies inside both triangles $\triangle x_1x_2x_3$ and $\triangle x_1'x_2'x_3'$ by their construction. In particular the midpoints (which are tangency points) of one triangle lie inside the other.

Without loss of generality assume that $x_1$ lies between $x_1'$ and $z_1$, which implies that $x_2$ lies between $x_2'$ and $z_1$. The midpoints $z_2'$ and $z_3'$ lie inside $\triangle x_1x_2x_3$, hence inside $\triangle z_3z_2x_3$ (otherwise $z_2'$ lies in quadrilateral $x_1x_2z_2z_3$ and $z_2$ is outside the angle $\angle x_1'x_2'z_2' = \angle x_1'x_2'x_3'$ while it should be in $\triangle x_1'x_2'x_3'$, the same argument for $z_3'$).

So, the segment $z_2'z_3'$ is inside $\triangle z_3z_2x_3$ and parallel to its side $z_2z_3$ since they are both midsegments (for $\triangle x_1'x_2'x_3'$ and $\triangle x_1x_2x_3$ respectively) and parallel to $l$. Then the length of $z_2'z_3'$ is at most length of $z_2z_3$, but, on the other hand, we have
$$|z_2'z_3'| = |z_1x_1'| > |z_1x_1| = |z_2z_3|$$
and get a contradiction.
\end{proof}

\begin{proposition}
\label{4-per-outer}
    For any $z \in \partial X$ there is at most one 4-periodic trajectory passing through $z$.
\end{proposition}

\begin{proof}
    Let $l$ be a tangent line with tangency point $z_1 = z$ and assume we have two 4-periodic orbits $(x_1, x_2, x_3, x_4)$ and $(x_1', x_2', x_3', x_4')$ with tangency points $z_1, z_2, z_3, z_4$ and $z_1, z_2', z_3', z_4'$ respectively. Since the domain is strongly convex, it lies inside both quadrilaterals $x_1x_2x_3x_4$ and $x_1'x_2'x_3'x_4'$. In particular the midpoints (which are tangency points) of one quadrilateral lie inside the other and form a parallelogram (e.g. $z_1z_2 \parallel x_1x_3 \parallel z_3z_4$ and $z_2z_3 \parallel  x_2x_4 \parallel  z_4z_1$). 


\begin{figure}[h]
    \centering
    \includegraphics[width=125mm]{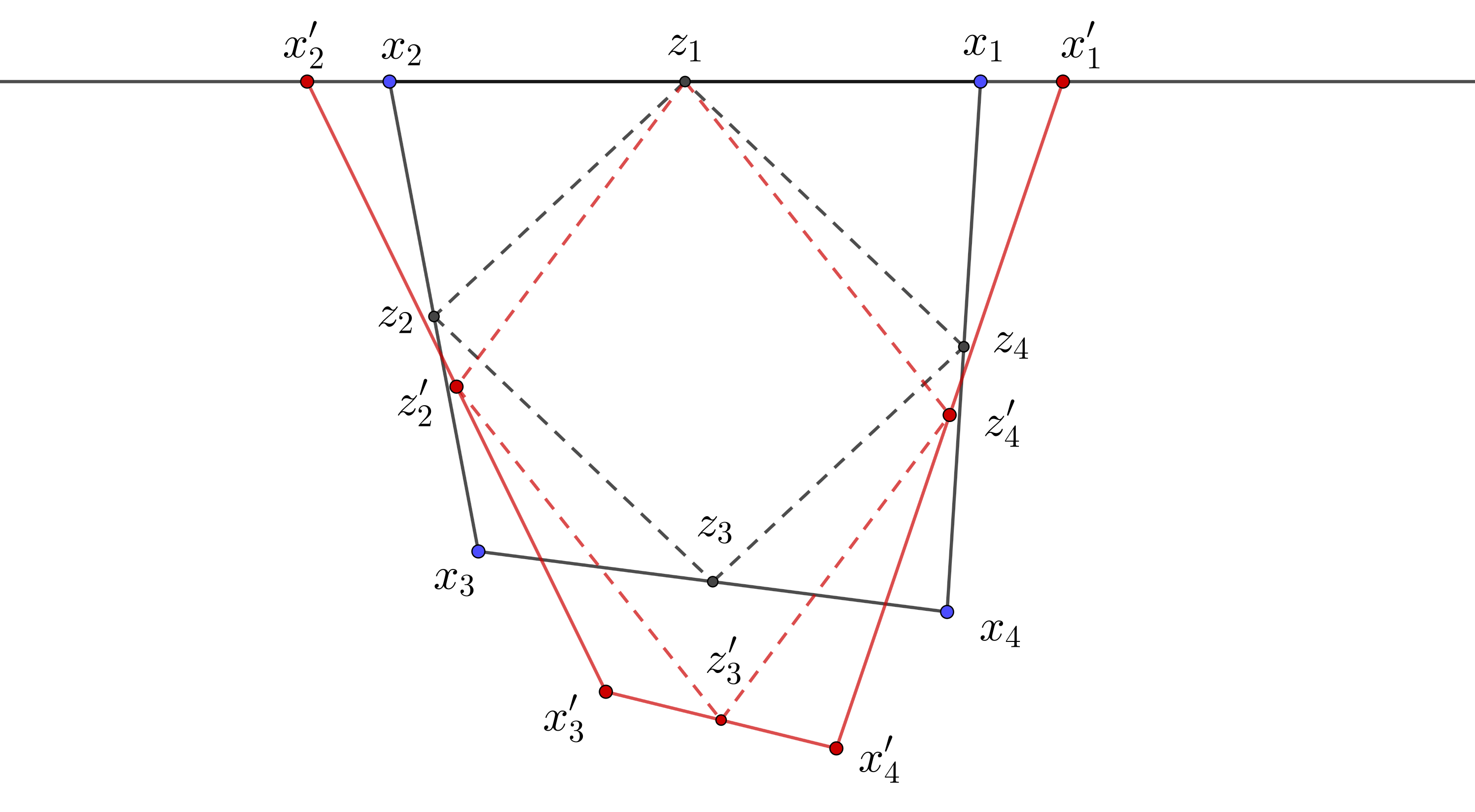}
    \caption{Two 4-periodic orbits}
    \label{two-4-periodic}
\end{figure}

    Without loss of generality assume that $x_1$ lies between $x_1'$ and $z_1$, which implies that $x_2$ lies between $x_2'$ and $z_1$. Then midpoints $z_2'$ and $z_4'$ lie inside the angle $\angle z_2z_1z_4$ (from strong convexity of $X$), hence the angle $\angle z_2'z_1z_4'$ is less than angle $\angle z_2z_1z_4$.

    Now we prove that a strongly convex body $X$ cannot contain two different inscribed parallelograms such that they have a common vertex $z$ and the angle at this vertex of one parallelogram lies inside the other one (strictly inside for both sides). Let $z$ be an origin. Denote vectors of sides from $z$ by $v_1, v_2$ for a parallelogram with a bigger angle and $w_1, w_2$ for a parallelogram with a smaller one. Take a basis $v_1, v_2$. Consider vectors from point $z$ which are inside the angle between $v_1$ and $v_2$ and with ends on $\partial X$. They all, except $v_1 +v_2$, have one of the coordinates bigger than 1 and the other bigger than 0 but smaller than 1 (from strong convexity of $X$).


\begin{figure}[h]
    \centering
    \includegraphics[width=60mm]{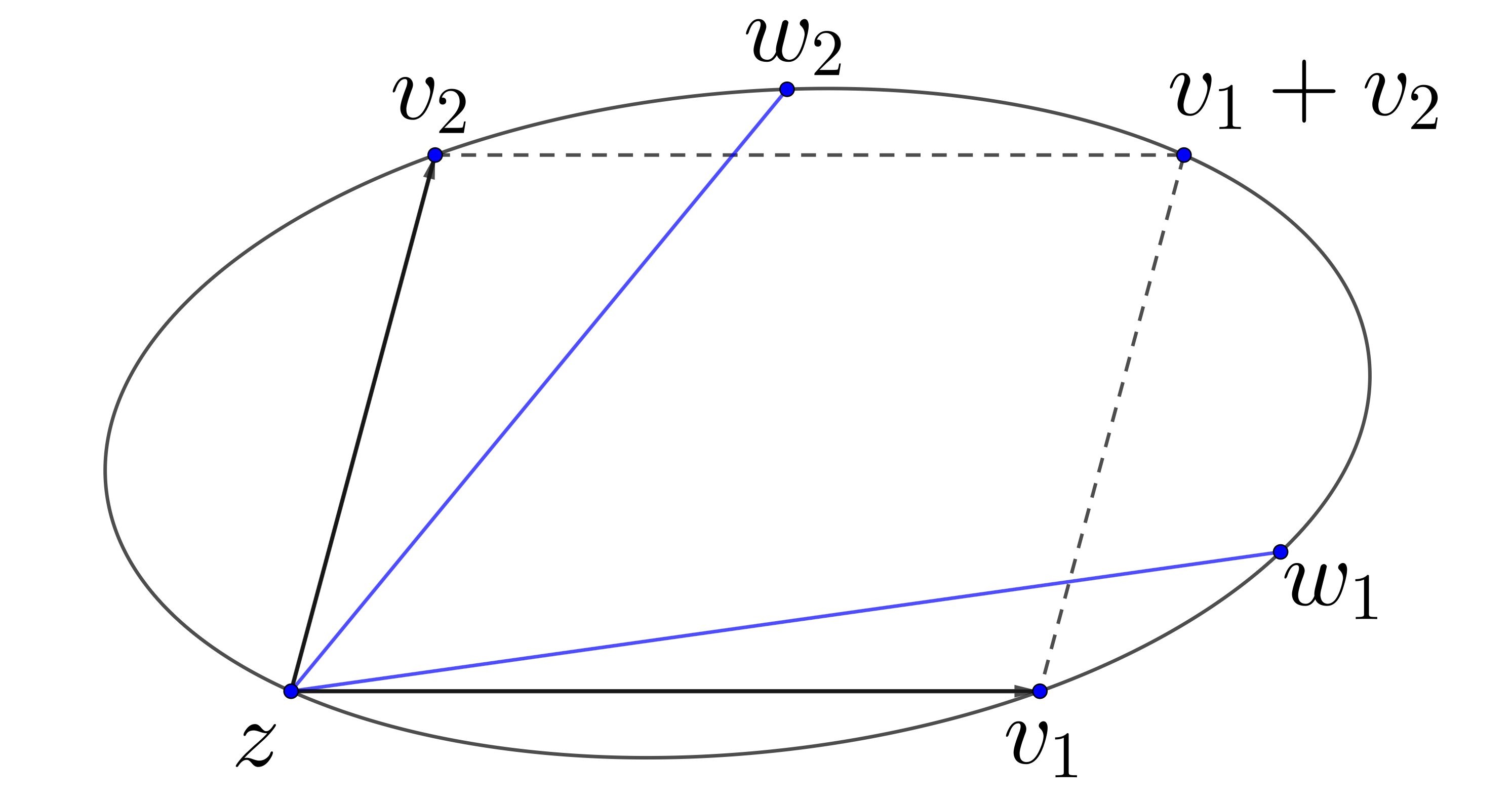}
    \caption{Parallelogram inside a convex domain}
    \label{2parallelograms}
\end{figure}

    If $w_1 = (a_1, b_1)$ and $w_2 = (a_2, b_2)$, then $w_1 + w_2 = (a_1 + a_2, b_1 + b_2)$ and we know it also lies on $\partial X$. Notice that $w_1+ w_2 \not= v_1 + v_2$ otherwise the parallelogram with a smaller angle is inside the parallelogram with a bigger angle. So, without loss of generality $a_1 + a_2 > 1$ and $b_1 + b_2 < 1$. Hence $b_1 < 1$ and $b_2 < 1$. Then we write 
    \[
    w_1 = \frac{a_1}{a_1 + a_2}(w_1 + w_2) + \frac{a_2b_1 - a_1b_2}{a_1 + a_2}v_2
    \]
    \[
    w_2 = \frac{a_2}{a_1 + a_2}(w_1 + w_2) - \frac{a_2b_1 - a_1b_2}{a_1 + a_2}v_2
    \]
    For one of them both coefficients are positive (they cannot be zero as $w_1$, $w_2$ are linearly independent and not parallel to $v_2$). Without loss of generality it holds for $w_1$. Then their sum is 
    \[
    \frac{a_1}{a_1 + a_2} + \frac{a_2b_1 - a_1b_2}{a_1 + a_2} < 1
    \]
    because $a_2b_1 - a_1b_2 < a_2b_1 < a_2$.
    So, we get that $w_1$ lies in the convex hull of $w_1+ w_2$ and $v_2$ and this contradicts that they all lie on $\partial X$.
\end{proof}

\begin{remark}
    One can prove a weaker statement that there are no nearby 4-periodic trajectories passing through the same tangency point just using the figure \ref{two-4-periodic}. In this case point $z_3$ is inside parallelogram $z_1z_2'z_3'z_4'$ which contradicts that it lies on $\partial X$.
\end{remark}

\begin{corollary}
\label{3-4-per}
    The sets of 3 and 4-periodic outer billiard orbits for a strongly convex domain in $\mathbb{R}^2$ with a smooth boundary have empty interiors. Moreover they have measure zero. 
\end{corollary}
\begin{proof}
    Assume there is an open set $U$ of $n$-periodic points. Take any $x_1 \in U$ and let $l$ be a tangent line from $x_1$ to $X$ and $z$ be a tangency point. Then $U \cap l$ contains an nonempty open interval of $n$-periodic points (any open ball with center in $x_1$ intersects with $l$ by an open interval) and hence there is an infinite set of $n$-periodic trajectories through $z$ which contradicts for $n = 3$ and $n = 4$ to propositions.

   Moreover, the set of 3 and 4-periodic points has measure zero. Indeed, let $P$ be the set of 3-periodic points, and assume by contradiction that it has Lebesgue measure $> 0$. There exists a point $x_1 \in P$ such that any neighborhood of $x_1$ intersects $P$ with positive measure. Since $x_1$ lies outside of the domain $X$, consider a tangent line to $\partial X$ going through $x_1$, more precisely let a point $z \in \partial X$ such that $T_z\partial X$ contains $x_1$.
   
Now consider a neighborhood $V$ of $x_1$. For any $z' \in \partial X$, endow $T_{z'}\partial X$ with the one-dimensional Lebesgue measure, and consider the intersection $T_{z'}\partial X \cap V$ as a subset of the latter line. If $V$ is sufficiently small, there is a neighborhood $W \subset \partial X$ of $z$ such that the family of sets $(T_{z'} \partial X \cap V)_{z' \in W}$ is a partition of $V$ . By Fubini’s theorem, there is a point $z' \in W$ such that $T_{z'} \partial X \cap V \cap P$ has non-zero measure as a subset of the line $T_{z'} \partial X$. In particular, $T_{z'} \partial X$ intersects $P$ infinitely many times which contradicts Proposition 2.1. The proof for period 4 is similar.  
   
\end{proof}

\begin{remark}
For higher periods the same idea does not work, one can construct a convex domain $X$ with two 5-periodic trajectories passing through the same tangency point. To construct such an example, it is enough to find two 5-gons which have the property that they have a common midpoint, and the other midpoints of each 5-gon lie inside the other. See figure \ref{5-gon_Ivrii}.
\end{remark}


\begin{figure}[h]
    \centering
    \includegraphics[width=90mm]{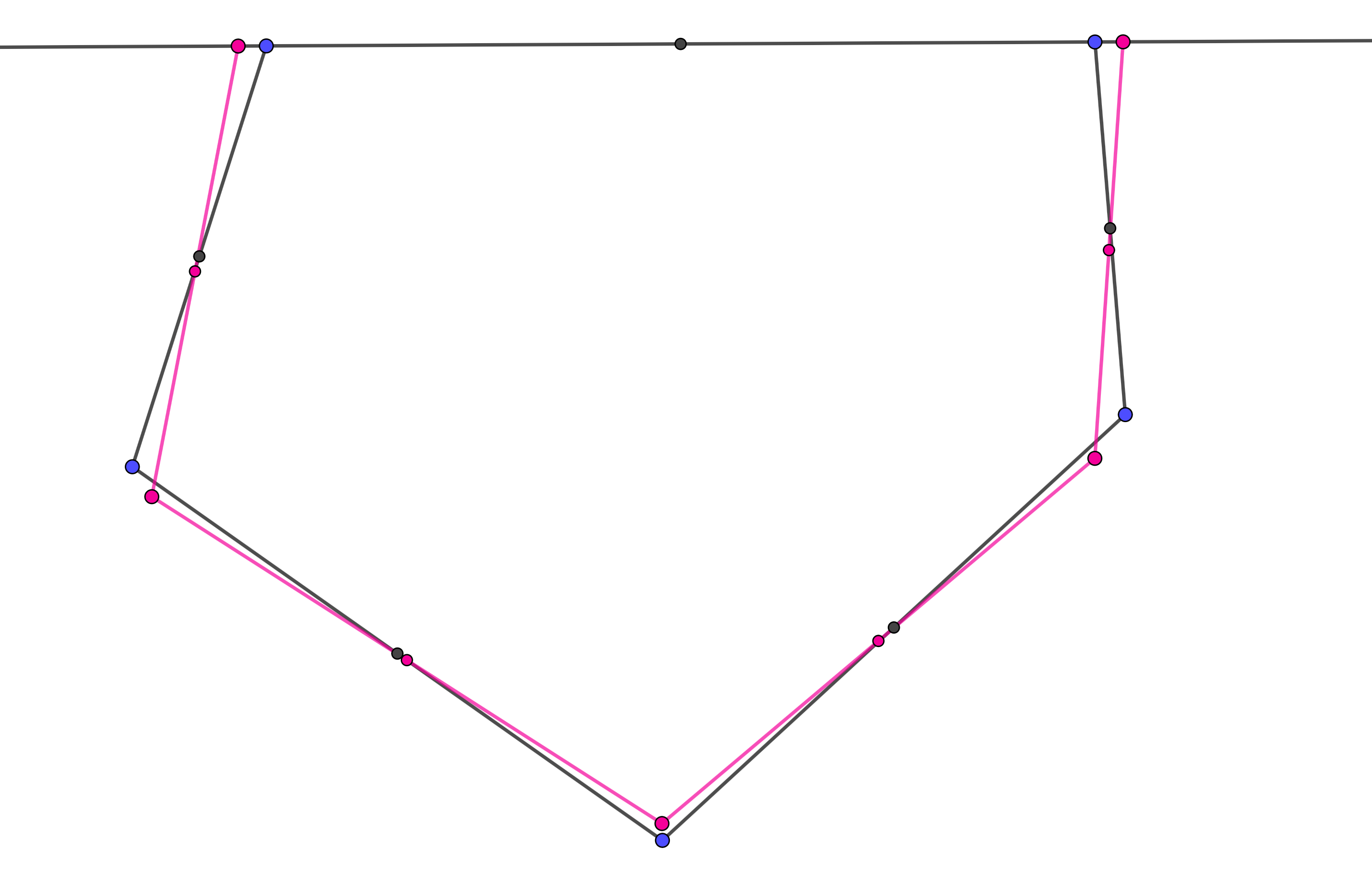}
    \caption{Two 5-periodic orbits}
    \label{5-gon_Ivrii}
\end{figure}

\begin{remark}
    This method also does not work for inner (conventional) billiards. As shown by Vorobetz \cite{vorobets94} (Lemma 1) one-dimensional family of 3-periodic orbits can pass through a boundary point. 
\end{remark}

\subsection{The cases of $(2n + 1, n)$ and $(2n, n-1)$}

 For any $n$-periodic outer billiard orbit we define a winding number as follows. Consider the orbit as a polygon with interior angles $\alpha_i$ and external angles $\beta_i = \pi - \alpha_i$. Then the winding number of this orbit is $m = \frac{1}{2\pi}\sum \beta_i$. Denote by $(n,m)$ an $n$-periodic orbit with winding number $m$. Notice that $\sum_{i = 1}^{n} \alpha_i = \pi (n - 2m)$, hence $0 < 2m < n$.

Assume we have an open set $U$ of $(n,m)$-periodic points, then for any $x \in U$: $F^n(x) = x$ and $dF^n|_x = Id$, where $F^n$ is the $n$-th iterate of $F$. The explicit form for the differential of the outer billiard map was found by Katok and Gutkin in \cite{gutkin_katok95}: for any point $x_0 \in \mathbb{R}^2 \setminus X$ with tangent line $l$ to $X$ and a tangency point $z \in \partial X$ choose a coordinate system at $x_0$ such that $x$-axis has the direction of vector $zx_0$ and $y$-axis as outer normal vector at $z$. Then in coordinates $\partial/\partial x$ and $\partial/\partial y$ the differential of the outer billiard map at $x_0$ has the following form:
\[
dF = \begin{pmatrix} -1 & -\frac{2\rho}{r} \\ 0 & -1 \end{pmatrix}
\]
where $\rho$ is a radii of curvature of $\partial X$ at $z$ and $r = |x_0z|$.


\begin{figure}[h]
    \centering
    \includegraphics[width=110mm]{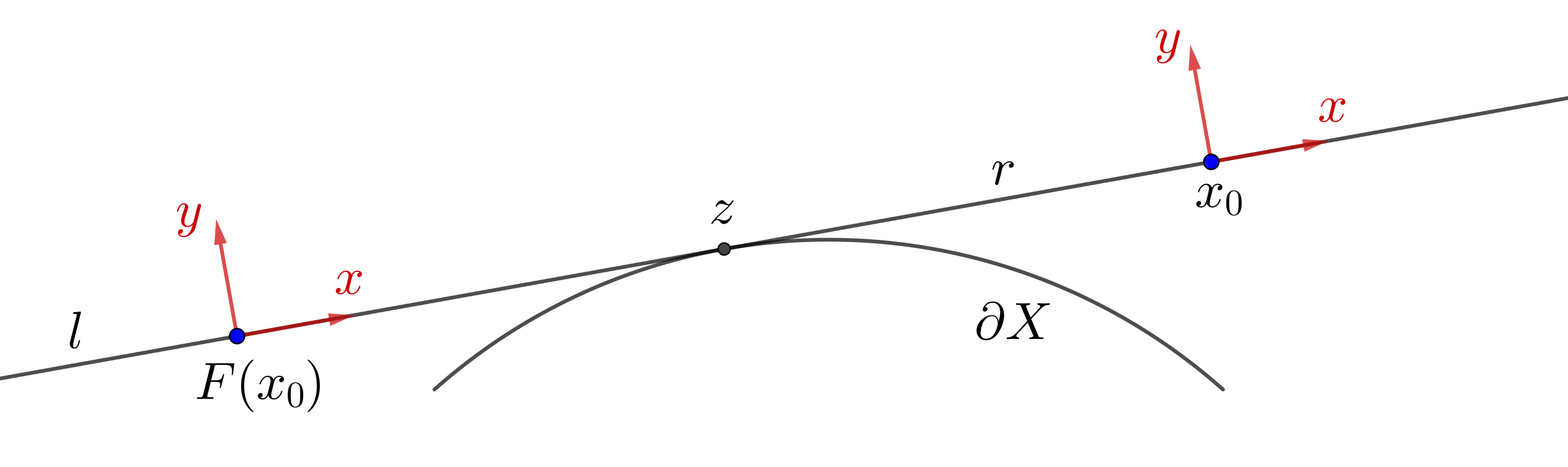}
    \caption{Coordinates for $dF$}
    \label{dF}
\end{figure}

Now consider an $n$-periodic orbit starting from $x_1$ as a polygon, denote its vertices as $x_1,..., x_n$ and tangency points as $z_1,..., z_n$ ($z_i$ is a midpoint between $x_i$ and $x_{i+1}$). Let $\rho_i$ be a radii of curvature of $X$ at point $z_i$, $\alpha_i$ be an interior angle between sides $x_ix_{i+1}$ and $x_{i+1}x_{i+2}$, $r_i = |x_iz_i|$. Then differential of $F^n$ at $x_1$ is 
\[
dF^n|_{x_1} = R(-\beta_n)dF|_{x_n}...R(-\beta_1)dF|_{x_1}
\]
where $R(\phi)$ is a matrix of rotation by angle $\phi$. Indeed, after each $dF|_{x_i}$ we should change coordinates to apply $dF|_{x_{i+1}}$ and it is given by the rotation by angle $-\beta_i = \pi + \alpha_i$. We rewrite this product denoting the matrix of differential $dF|_{x_i}$ by $-A_i$ and having $R(-\beta_i) = - R(\alpha_i)$:
\[
dF^n|_{x_1} = R(\alpha_n)A_n...R(\alpha_1)A_1
\]
A matrix of the form $ \begin{pmatrix} 1 & s \\ 0 & 1 \end{pmatrix}$ is called a shear matrix and corresponds to a horizontal shear transformation. So, the differential of the $n$-th iterate is equal to the product of rotations and shear matrices with positive elements $s = \frac{2\rho}{r}$, as the curvature at each boundary point is positive. It appears that this product cannot be identity for some sets of angles.

\begin{lemma}
\label{shear-rotation}
    Let $A_i = \begin{pmatrix} 1 & s_i \\ 0 & 1 \end{pmatrix}$ be a shear matrix with positive $s_i$ and assume that we have for $0 < \alpha_i < \pi$ $$R(\alpha_n)A_n...R(\alpha_1)A_1 = Id.$$ 
    Then $\sum_{i = 1}^n \alpha_i > 2\pi$.
\end{lemma}
\begin{proof}
    Consider all halflines $l$ through the origin and define a quasi-direction of a halfline as follows:
    \[
    q(l) =
    \begin{cases}
        0, & \text{if $l$ has a direction of the positive $x$-axis;} \\
        1, & \text{if $l$ lies in the upper halfplane;}\\
        2, & \text{if $l$ has a direction of the negative $x$-axis;}\\
        3, & \text{if $l$ lies in the lower halfplane.}
    \end{cases}
    \]
    Each shear matrix of above form rotates a halfline clockwise by some angle and preserves its quasi-direction. Each rotation rotates halflines counterclockwise by angles $0 < \alpha_i < \pi$ and does not decrease quasi-direction except for the cases when the line goes from 3 to 0 or 1. Consider a halfline $l_0$ spanned by the vector $(1,0)^T$. After the first shear and the first rotation $l_0$ has a quasi-direction 1. The second shear preserves the quasi-direction of $l_0$ equaled 1 and rotates $l_0$ clockwise by nonzero angle $\phi_2$. Having  $R(\alpha_n)A_n...R(\alpha_1)A_1 = Id$ means that after all shears and rotations $l_0$ should return to quasi-direction 0. But to return back to 0, $l_0$ should be rotated at least for one whole cycle of quasi-directions 0,1,2,3,0. If  $\sum_{i = 1}^n \alpha_i \le 2\pi$ then $\sum_{i = 1}^n \alpha_i - \sum_{i = 2}^n\phi_i < 2\pi$ and the angle of whole rotation of $l_0$ is less then $2\pi$.
\end{proof}

The Theorem \ref{star_orbits_theorem} follows from the Lemma \ref{shear-rotation} since any $(n, m)$-orbit has the sum of angles $\sum \alpha_i = \pi (n - 2m)$ which is equal to $\pi$ for $(2n + 1, n)$-periodic orbit and $2\pi$ for $(2n, n - 1)$-periodic orbit.

The idea of proving that the differential of $n$-th iterate of the outer billiard map cannot be identity apparently does not work for a bigger sum of angles. One can check that for any $n \ge 5$, choosing $\alpha_1 = ... = \alpha_n = \frac{(n-2)\pi}{n}$ and $s_1 = ...= s_n = 4\cot{\frac{2\pi}{n}}$ we get:
\[
\left(\begin{pmatrix} \cos{\frac{(n-2)\pi}{n}} & -\sin{\frac{(n-2)\pi}{n}} \\ \sin{\frac{(n-2)\pi}{n}} & \cos{\frac{(n-2)\pi}{n}} \end{pmatrix}\begin{pmatrix} 1 & 4\cot{\frac{2\pi}{n}} \\ 0 & 1 \end{pmatrix}\right)^n = \begin{pmatrix} 1 & 0 \\ 0 & 1 \end{pmatrix}.
\]

For example, for $n = 8$, we choose $\alpha_i = 3\pi/4$ and $s_i = 4$ and have:
\[
\left(\begin{pmatrix} \cos{3\pi/4} & -\sin{3\pi/4} \\ \sin{3\pi/4} & \cos{3\pi/4} \end{pmatrix}\begin{pmatrix} 1 & 4 \\ 0 & 1 \end{pmatrix}\right)^8 = \begin{pmatrix} 1 & 0 \\ 0 & 1 \end{pmatrix}.
\]

\begin{remark}
    This example does not work for $n=3$ and $n=4$ since $\cot{\frac{2\pi}{3}} < 0$ and $\cot{\frac{2\pi}{4}} = 0$.
\end{remark}

\subsection{Example of an 8-periodic union of two segments}

For $n = 8$ it is possible not only to find a point with the identical differential of the eighth iterate of the outer billiard map but also to have two intersected segments of 8-periodic points as in the figure $\ref{8-periodic}$.


\begin{figure}[h]
    \centering
    \includegraphics[width=125mm]{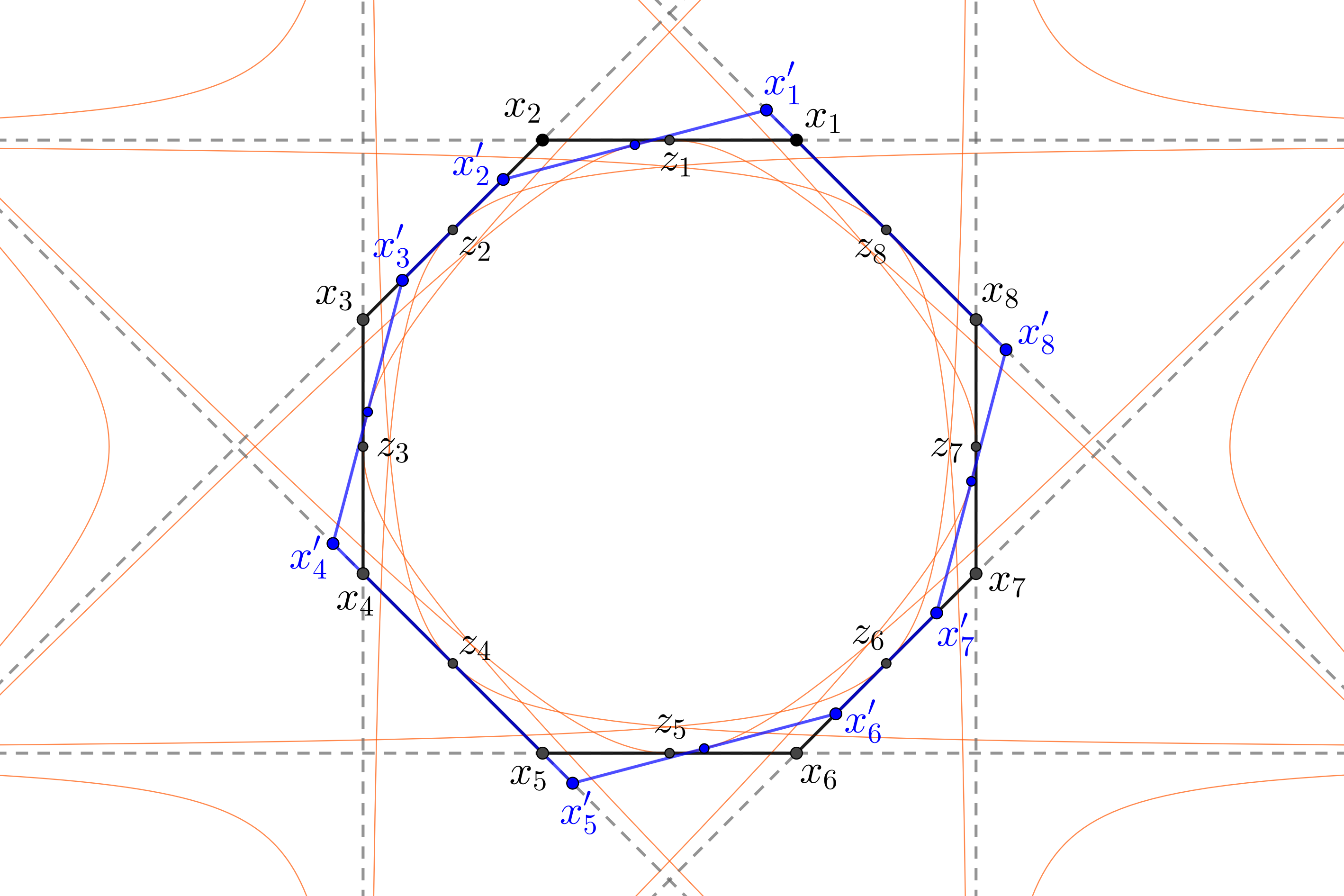}
    \caption{Two 8-periodic orbits}
    \label{8-periodic}
\end{figure}

Consider a regular 8-gon with vertices $x_1,...x_8$ and midpoints $z_1,...z_8$ ($z_i$ is a midpoint between $x_i$ and $x_{i+1}$). For each $z_i$ draw a hyperbola $h_i$ which touches $x_ix_{i+1}$ and has asymptotes $x_{i-1}x_{i}$ and $x_{i+1}x_{i+2}$. Each tangent segment to hyperbola with ends on the asymptotes is bisected by a tangency point \cite{wells-1991}. So, $h_i$ touches $x_ix_{i+1}$ at point $z_i$. Keep small pieces of each $h_i$ (intersection of the hyperbola with a small ball centered at $z_i$) and connect them together such that it turns to a closed smooth strongly convex curve. Now, one can check that there is a neighborhood $U$ of $x_1$ such that each point from segments $U \cap x_1x_2$ and $U \cap x_1x_8$ are 8-periodic. 

Indeed, if we consider a point $x_1'$ close to $x_1$ on the line $x_1x_8$, it reflects with respect to $h_1$ to point $x_2'$ on the line $x_2x_3$ (from the hyperbola property described above) and then reflects with respect to $h_2$ (same as with respect to $z_2$) to the point $x_3'$ on the line $x_2x_3$, and then $x_3'$  reflects with respect to $h_3$ to $x_4'$ on $x_4x_5$. Then $x_1x_1' = x_4x_4'$ from a symmetry. If we continue we finally get that $x_8'$ reflects with respect to $z_8$ to $x_1$. Similarly, we get for the points close to $x_1$ on the line $x_1x_2$.

\subsection{Higher dimensions}

For outer billiards in higher dimensions one should consider a convex body $X$ in the standard symplectic space $(\mathbb R^{2n}, \omega)$. At any point $z$ on the boundary $\partial X$ there is a one-dimensional \emph{characteristic} direction $JN_z \subset  \ker \omega|_{T_z\partial X}$, where $N_z$ is an outer normal vector at point $z \in \partial X$ and $J = \begin{pmatrix} 0 & I \\ -I & 0 \end{pmatrix}$ is a standard complex structure. Then for any point $x$ outside of $X$ there is a unique point $z$ on the boundary such that $xz$ is a tangent line with a characteristic direction. Then one can define an outer billiard map as before mapping $x$ to $y$ on the tangent line with characteristic direction such that the tangency point is a midpoint between $x$ and $y$. For a detailed introduction to higher outer (dual) billiards the reader is referred to \cite{tabach93}.

\begin{question}
    Is it possible that for some $z \in \partial X$ there is at least two 3-periodic orbits passing through $z$?
\end{question}
Another question connected with periodic points of higher dimensional outer billiards but slightly different from the Ivrii conjecture:
\begin{question}
    Is it possible that for some $z \in \partial X$ all outer billiard trajectories through $z$ are periodic?
\end{question}
For planar case the answer is no because of the existence of invariant curves with irrational rotation numbers far from $\partial X$,and so the invariant curves intersect with the tangent line through $z$ implying that their intersection points are not periodic. For higher dimensions the existence of invariant hypersurfaces is unknown. See \cite{moser73} for more details about invariant curves for outer billiards.

\section{Ivrii's conjecture for symplectic billiards}

Symplectic billiard was introduced by Albers and Tabachnikov \cite{albers-tabach2017} as a new inner billiard whose generating function is area. The precise definition of the symplectic billiard map is the following: for a strictly convex domain $X \subset \mathbb{R}^2$  the ray $xy$ ($x,y \in \partial X$) reflects to the ray $yz$ if $xz$ is parallel to the tangent line of $\partial X$ at point $y$.


\begin{figure}[h]
    \centering
    \includegraphics[width=70mm]{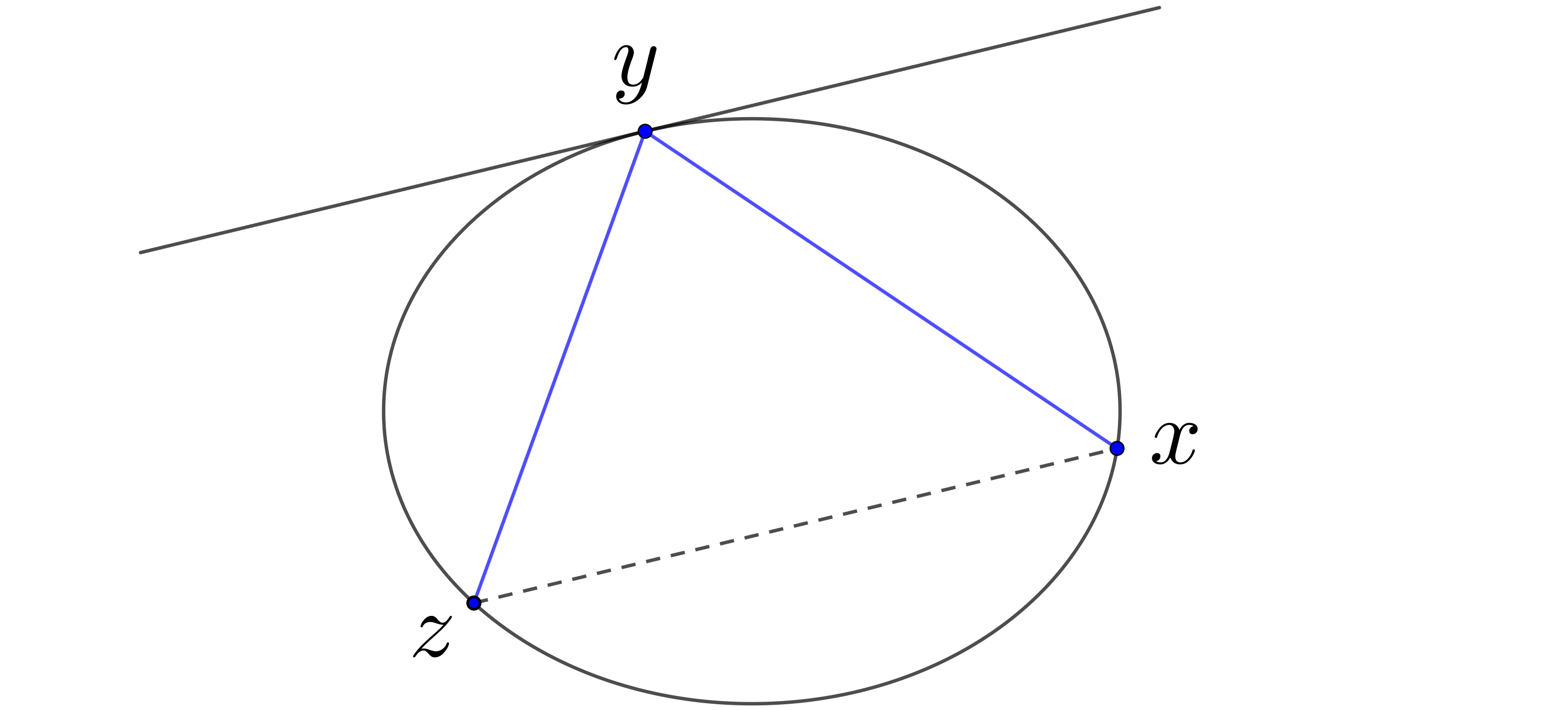}
    \caption{Symplectic billiard reflection law}
    \label{fig:enter-label}
\end{figure}

To define symplectic billiard in higher dimensions one needs the choice of a tangent line. Similarly to the outer billiard we consider the standard symplectic space $(\mathbb{R}^{2n},\omega)$ and a characteristic direction $JN_z$ for each boundary point $z \in \partial X$ (where $N_z$ is an outer normal vector at point $z \in \partial X$ and $J = \begin{pmatrix} 0 & I \\ -I & 0 \end{pmatrix}$). Then one can well define symplectic billiard.

Albers and Tabachnikov \cite{albers-tabach2017} proved using EDS that for planar symplectic billiards the set of 3-periodic points and 4-periodic points has an empty interior. Now we give easier geometric proofs for these periods in both planar and higher-dimensional cases. For period 3 we use the results for outer billiards. The idea of our proof for period 4 is similar to the proofs of Propositions \ref{3-per-outer} and \ref{4-per-outer} that through any point A on the boundary of a strictly convex body there is at most one 4-periodic orbit. This implies that the set of 4-periodic orbits has an empty interior.

\subsection{3-periodic orbits}

Let $X$ be a convex domain with a smooth boundary. Its 3-periodic symplectic billiard trajectories are given by the triangles $ABC$ such that each side is parallel to the characteristic line through an opposite point (e.g. $JN_A$ parallel to $BC$). The three characteristic lines through points $A$, $B$, $C$ lie in the same plane and intersect at points $A_1$, $B_1$, $C_1$. Then $A_1B_1C_1$ is an outer billiard orbit for $X$. Indeed, from the parallelism of tangent segments $A_1B_1$, $B_1C_1$, $C_1A_1$ with sides of the triangle $ABC$ one gets 3 parallelograms: $ABA_1C$, $ABCB_1$, and $AC_1BC$. Clearly, that $A$, $B$ and $C$ are midpoints of $B_1C_1$, $A_1C_1$ and $A_1B_1$ respectively. On the other hand, for any 3-periodic outer billiard trajectory $A_1B_1C_1$ we have that the triangle consisting of the midpoints of $A_1B_1C_1$ is a 3-periodic symplectic billiard trajectory.  So, we have a one-to-one correspondence between 3-periodic outer billiard trajectories and 3-periodic symplectic billiard trajectories, hence one-to-one (continuous) correspondence for 3-periodic orbits.

\begin{lemma}
    The set of 3-periodic orbits of symplectic billiard has an empty interior.
\end{lemma}
\begin{proof}

Since we have a one-to-one correspondence between 3-periodic outer billiard orbits and 3-periodic symplectic billiard orbits, the planar case follows from corollary \ref{3-4-per}.

In higher dimensions ($n >1$) if the set of 3-periodic points of symplectic billiard has nonempty interior then from the one-to-one correspondence described above we have an injective continuous map from some open set of 3-periodic points $U$ (which is $(4n-2)$-dimensional manifold as it is an open set of the phase space for symplectic billiard) to the phase space of outer billiard (exterior of the domain which is $2n$-dimensional manifold). Hence, for $n > 1$ we have an injective continuous map from a manifold of higher dimension to a manifold of lower dimension, which is impossible. Thus, the set of 3-periodic points has an empty interior which implies empty interior for 3-periodic orbits.

\end{proof}

\subsection{4-periodic orbits}

\begin{lemma}
    There is at most one 4-periodic symplectic billiard trajectory through a given boundary point.
\end{lemma}
\begin{proof}
    We give a general proof which works for any dimension (both planar and higher). Notice that for any 4-periodic trajectory $ABCD$ ($A$, $B$, $C$, $D$ are points on the boundary of the body) one has that the characteristic line through point  $A$, the characteristic line through point $C$ and the segment $BD$ are parallel from the definition of the symplectic billiard map. Similarly, the characteristic line through point $B$, the characteristic line through point $D$ and the segment $AC$ are parallel. From strict convexity of the body for any point $A$ there is exactly one point $C$ with the opposite characteristic direction and there are unique points $B$ and $D$ with characteristic directions parallel to $AC$. Thus there is at most one 4-periodic trajectory through point $A$. 
\end{proof}

\section*{Acknowledgments}
The author thanks Sergei Tabachnikov, Vadim Zharnitsky, Peter Albers, Maciej Wojtkowski and Alexei Glutsyuk for fruitful discussions during the Mathematical billiard program in the Simons Center for Geometry and Physics. The author thanks Piotr Laskawiec for useful comments and the unknown referees for useful remarks. The author is also grateful to the hospitality of the Heidelberg University and the Simons Center.









\end{document}